\documentclass[11pt]{amsart}

\usepackage{amsfonts}
\usepackage{amssymb}
\usepackage{amsthm}
\usepackage{amsmath}
\usepackage{hyperref}
\usepackage[top=3.5cm, bottom=3.5cm, left=2.5cm, right=2.5cm]{geometry}
\usepackage[capitalise]{cleveref}
\usepackage{tikz}
\usepackage{pgf}
\usepackage{graphicx}

\usepackage[latin1]{inputenc}

\newtheorem{prop}{Proposition}
\newtheorem{theorem}[prop]{Theorem}
\newtheorem{lemma}[prop]{Lemma}

\theoremstyle{definition}
\newtheorem{definition}[prop]{Definition}

\theoremstyle{remark}

\newtheorem*{remark*}{Remark}

\newcommand{\N}{\mathbb{N}}
\newcommand{\Z}{\mathbb{Z}}

\newcommand{\E}{\mathbb{E}}
\newcommand{\K}{\mathbb{K}}
\newcommand{\F}{\mathbb{F}}

\newcommand{\eps}{\varepsilon}

\title{A brief introduction to approximate groups}
\date{}
\author{Matthew C. H. Tointon}
\address{Pembroke College, University of Cambridge, CB2 1RF, United Kingdom}
\email{mcht2@cam.ac.uk}
\thanks{The author is the Stokes Research Fellow at Pembroke College, University of Cambridge.}

\begin{document}
\maketitle
\begin{abstract} We give a brief introduction to the notion of an \emph{approximate group} and some of its numerous applications.
\end{abstract}

\tableofcontents

This is a translation of the author's article `Raconte-moi\ldots les groupes approximatifs', which appeared in the \textit{Gazette des math\'ematiciens} in April 2019 \cite{raconte-moi}.

\section{Approximately closed sets}
Mathematicians are used to the notion of a subgroup of a group $G$ as a subset containing the identity that is closed under taking products and inverses. However, it turns out that there are also circumstances in which we encounter subsets that are merely `approximately closed'. Such sets arise in the study of \emph{polynomial growth} in geometric group theory (which in turn has links to isoperimetric inequalities and random walks) and in the construction of \emph{expander graphs} (which are important objects in computer science), but there are also numerous other examples.

A priori, there are several different ways to define approximate closure. One of these is the notion of a set of \emph{small doubling}, with which we commence our discussion; another is the notion of an \emph{approximate subgroup}, which we present in detail in Section \ref{sec:app.grp}. We shall see that these two notions are intimately linked.

We start by giving one interpretation of the phrase `approximately closed'. Given subsets $A,B$ of a group $G$, we set $AB=\{ab:a\in A,b\in B\}$ and $A^{-1}=\{a^{-1}:a\in A\}$. We also set $A^2=AA$, $A^3=AAA$, and so on. For additive abelian groups we write instead $A+B$, $-A$, $2A=A+A$, $3A=A+A+A$ and so on. To say that a finite subset $A$ is closed under the group operation is then to say that $A^2=A$. One property that could be interpreted as being an \emph{approximate} version of closure is thus that $A^2$ is not too much bigger than $A$ (we will discuss very briefly in Section \ref{sec:app.grp} a possible extension to infinite subsets).

Let us consider for a moment the extreme values that $|A^2|$ can take. It is clear that $|A^2|\ge|A|$, with equality when $A$ is a finite subgroup, for example. On the other hand, it is clear that
$|A^2|\le|A|^2$, with equality if $A=\{x_1,\ldots,x_r\}$ and $G$ is the free group on the generators $x_i$.

Although it is extremal, the case in which $|A^2|$ is comparable  to $|A|^2$ should not be thought of as atypical. Indeed, there is a fairly general phenomenon whereby if $A$ is a suitably defined random set of size $k$ inside some group then $\E[|A^2|]\ge ck^2$ for some constant $c$ depending on the context. For example, if $A$ is chosen uniformly from the interval $\{1,\ldots,n\}\subset(\Z,+)$ with $n$ much larger than $k$ then one can essentially take $c=\frac{1}{2}$ \cite[Proposition 2.1.1]{book}. This suggests that a condition of the type $|A^2|=o(|A|^2)$ is a stong constraint on the set $A$. We will consider this condition in its strongest form, in which
\begin{equation}\label{eq:sm.doub}
|A^2|\le K|A|
\end{equation}
for a given $K\ge1$.

\begin{definition}A set $A$ satisfying \eqref{eq:sm.doub} is called a \emph{set of doubling at most $K$}, or simply a \emph{set of small doubling}. The quantity $|A|^2/|A|$ is called the \emph{doubling constant} of $A$. Similarly, a set $A$ satisfying $|A^3|\le K|A|$ is called a \emph{set of tripling at most $K$}, or simply a \emph{set of small tripling}. The quantity $|A|^3/|A|$ is called the \emph{tripling constant} of $A$.
\end{definition}

Since the inequality \eqref{eq:sm.doub} is in some sense the opposite of what we would expect from a random set, it is reasonable to suppose that a set of small doubling should possess a certain amount of `structure'. One of the principal goals of the theory of approximate groups is to describe this structure. In this article we give a brief overview of this theory; for more details, and for a more complete bibliography, the reader can consult the author's book \cite{book}.

\bigskip

We will often assume that the set $A$ contains the identity and is \emph{symmetric}, which is to say closed under taking inverses. For the majority of the results we present this is not a necessary hypothesis, but it simplifies the exposition and the notation.

\section{First examples}\label{sec:exemples}
A trivial family of examples of sets of small doubling is given by small sets: if $|A|\le K$ then of course $A$ satisfies $|A^2|\le K|A|$. We will therefore focus on sets of size significantly larger than $K$. 
Finite subgroups also give easy examples of sets of small doubling. Note also that if a set $A_0$ has doubling constant at most $K$, and $A$ is a subset of $A_0$ of \emph{density} at most $\alpha\in[0,1]$ (which is to say that $|A|\ge\alpha|A_0|$), then we have
\[
|A^2|\le|A_0^2|\le K|A_0|\le\frac{K}{\alpha}|A|,
\]
and so the doubling constant of $A$ is at most $K/\alpha$. Thus, if $A$ is sufficiently dense in some set of small doubling $A_0$ then $A$ is also a set of small doubling. In particular, if $H$ is a finite subgroup of $G$ and the density in $H$ of some subset $A\subset H$ is at least $1/K$ then the doubling constant of $A$ is at most $K$. Freiman showed that for small enough $K$ this essentially exhausts all of the examples of sets of doubling $K$. More precisely, he showed that if $|A^2|<\frac{3}{2}|A|$ then $A^2$ is a coset of a finite subgroup (see \cite[Theorem 2.2.1]{book}).

We now consider a more interesting example. Note first of all that if $B\subset\Z^d$ is a `box' of the form
\[
B=\{x\in\Z^d:|x_i|\le L_i\text{ for $i=1,\ldots,d$}\}
\]
for some $L_i\in\N$ then
\begin{equation}\label{eq:sm.doub.B}
|B+B|\le2^d|B|
\end{equation}
regardless of the values taken by the $L_i$. Boxes  in $\Z^d$ are thus sets of small doubling. It is also easy to check that their homomorphic images are also sets of small doubling. To see this, first note that such a box $B$ satisfies a stronger property than \eqref{eq:sm.doub.B}, in that there exists a set $X$ satisfying $|X|=2^d$ such that
\begin{equation}\label{eq:cov.B}
B+B\subset B+X,
\end{equation}
as illustrated in the following diagram.
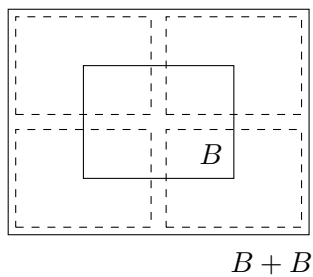
\vspace{1em}
\begin{center}
\begin{tikzpicture}

\draw (-2,-1.5) -- (2,-1.5) -- (2,1.5) -- (-2,1.5) -- cycle;
\draw (1.5,-1.6) node[below] {$B+B$};
\draw (0.7,-0.7) node[above] {$B$};

\draw (-1,-0.75) -- (1,-0.75) -- (1,0.75) -- (-1,0.75) -- cycle;

\begin{scope}[shift={(0,1.5)}]
\begin{scope}[shift={(2,0)}]
\draw[dashed] (-3.9,-2.9) -- (-2.1,-2.9) -- (-2.1,-1.6) -- (-3.9,-1.6) -- cycle;
\end{scope}

\begin{scope}[shift={(4,0)}]
\draw[dashed] (-3.9,-2.9) -- (-2.1,-2.9) -- (-2.1,-1.6) -- (-3.9,-1.6) -- cycle;
\end{scope}
\end{scope}

\begin{scope}[shift={(0,3)}]
\begin{scope}[shift={(2,0)}]
\draw[dashed] (-3.9,-2.9) -- (-2.1,-2.9) -- (-2.1,-1.6) -- (-3.9,-1.6) -- cycle;
\end{scope}

\begin{scope}[shift={(4,0)}]
\draw[dashed] (-3.9,-2.9) -- (-2.1,-2.9) -- (-2.1,-1.6) -- (-3.9,-1.6) -- cycle;
\end{scope}
\end{scope}

\end{tikzpicture}
\end{center}
This means that if $G$ is an abelian group and $\pi:\Z^d\to G$ is a homomorphism then $\pi(B)+\pi(B)\subset\pi(X)+\pi(B)$. In particular, $|\pi(B)+\pi(B)|\le2^d|\pi(B)|$, and so $\pi(B)$ has small doubling.

A homomorphic image of a box such as $B$ is called a \emph{progression}. More precisely, if $x_1,\ldots,x_d$ are elements of an abelian group and $L_1,\ldots,L_d\in\N$ then we set
\[
P=P(x;L)=\{\ell_1x_1+\cdots+\ell_dx_d:|\ell_i|\le L_i\}.
\]
We call $P$ a \emph{progression}, and we call $d$ the \emph{rank} or the \emph{dimension} of $P$. For example, in the following diagram we illustrate the progression $P(9,2;2,1)\subset\Z$, viewed as $\pi(\Z^2\cap([-2,2]\times[-1,1]))$ with $\pi:\Z^2\to\Z$ defined by $\pi(1,0)=9$ and $\pi(0,1)=2$.

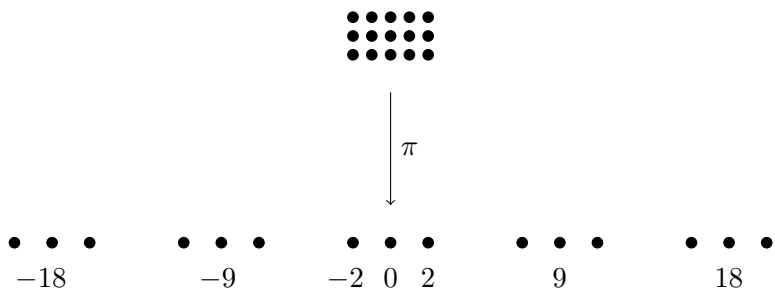
\vspace{1em}
\begin{center}
\begin{tikzpicture}

\filldraw (-0.5,3) circle (2pt);
\filldraw (-0.25,3) circle (2pt);
\filldraw (0,3) circle (2pt);
\filldraw (0.25,3) circle (2pt);
\filldraw (0.5,3) circle (2pt);

\filldraw (-0.5,2.75) circle (2pt);
\filldraw (-0.25,2.75) circle (2pt);
\filldraw (0,2.75) circle (2pt);
\filldraw (0.25,2.75) circle (2pt);
\filldraw (0.5,2.75) circle (2pt);

\filldraw (-0.5,2.5) circle (2pt);
\filldraw (-0.25,2.5) circle (2pt);
\filldraw (0,2.5) circle (2pt);
\filldraw (0.25,2.5) circle (2pt);
\filldraw (0.5,2.5) circle (2pt);

\draw[->] (0,2) -- node[right] {$\pi$} (0,0.5);

\filldraw (-5,0) circle (2pt);
\filldraw (-4.5,0) circle (2pt);
\filldraw (-4,0) circle (2pt);
\draw (-4.65,-0.2) node[below] {$-18$};

\filldraw (-2.75,0) circle (2pt);
\filldraw (-2.25,0) circle (2pt);
\filldraw (-1.75,0) circle (2pt);
\draw (-2.3,-0.2) node[below] {$-9$};

\filldraw (-0.5,0) circle (2pt);
\filldraw (0,0) circle (2pt);
\filldraw (0.5,0) circle (2pt);
\draw (-0.6,-0.2) node[below] {$-2$};
\draw (0,-0.2) node[below] {$0$};
\draw (0.5,-0.2) node[below] {$2$};

\filldraw (1.75,0) circle (2pt);
\filldraw (2.25,0) circle (2pt);
\filldraw (2.75,0) circle (2pt);
\draw (2.25,-0.2) node[below] {$9$};

\filldraw (4,0) circle (2pt);
\filldraw (4.5,0) circle (2pt);
\filldraw (5,0) circle (2pt);
\draw (4.5,-0.2) node[below] {$18$};

\end{tikzpicture}
\end{center}
To explain the term \emph{progression}, note that if the rank of $P$ is $1$ then $P$ is an arithmetic progression.

We have just seen that subgroups, progressions of bounded rank, and their dense subsets are all examples of sets of small doubling. The following remarkable theorem, due to Freiman in the case $G=\Z$ and Green and Ruzsa in the general case, shows that these are essentially the only examples in an abelian group.

\begin{theorem}[Green--Ruzsa]\label{thm:fgr}
Let $G$ be an abelian group, and suppose that $A\subset G$ is a finite subset satisfying $|A+A|\le K|A|$. Then there exist a finite subgroup $H$ and a progression $P$ of rank at most $r(K)$ such that $A$ is a subset of $H+P$ of density at least $\delta(K)$.
\end{theorem}
The proof is largely Fourier analytic, and gives explicit bounds on $r(K)$ and $\delta(K)$. Optimising these bounds continues to be an area of active research.

\section{Pl\"unnecke's inequalities and Ruzsa's covering lemma}
The proof of Theorem \ref{thm:fgr} is too long to be included in this article, but we will illustrate two fundamental tools from the proof by considering the following special case.
\begin{prop}[Ruzsa]\label{prop:vector}
Let $m\in\N$, and let $G$ be an abelian group in which each element has order at most $m$ (such as $G=(\Z/m\Z)^n$ for some $n\in\N$). Suppose that $A$ is a finite symmetric subset of $G$ such that $|A+A|\le K|A|$. Then $A$ is a subset of density at least $1/(m^{K^4}K)$ in some finite subgroup of $G$.
\end{prop}

The first tool we present is \emph{Pl\"unnecke's inequalities}, which were first proved by Pl\"unnecke, then rediscovered and generalised by Ruzsa, and finally proved much more simply by Petridis.
\begin{prop}[Pl\"unnecke--Ruzsa]\label{prop:plun}
Let $G$ be an abelian group and suppose that $A$ is a finite subset satisfying $|A+A|\le K|A|$. Then $|mA-nA|\le K^{m+n}|A|$ for every $m,n\in\N$.
\end{prop}
We will soon see concretely the role that this result plays in the proof of Proposition \ref{prop:vector}, but before that let us give a brief heuristic discussion of why one might expect such a result to be useful. First, note that if $H$ is a subgroup then $mH=H$ for every $m\in\N$, a property that we use often without even thinking. Proposition \ref{prop:plun} says that a set of small doubling satisfies an approximate version of this property: if $A$ is a finite set satisfying $|A+A|\le K|A|$ then, for every $m\in\N$, on the one hand the set $mA$ is not much bigger than $A$, and on the other hand it is also of small doubling, in the sense that $|mA+mA|\le K^{2m}|A|\le K^{2m}|mA|$.

Another important tool featuring in the proof of Proposition \ref{prop:vector} is the so-called `covering lemma' of Ruzsa. We present a slightly simplified version of it here; see \cite[Lemma 5.1]{bgt} for a more general statement.
\begin{lemma}[Ruzsa]\label{lem:cover}
Suppose $A$ is a finite symmetric subset of a group $G$ such that $|A^4|\le K|A|$. Then there exists $X\subset G$ of size at most $K$ such that $A^3\subset XA^2$.
\end{lemma}
\begin{proof}
Let $X\subset A^3$ be maximal such that the subsets $xA$ with $x\in X$ are disjoint, noting that $|XA|=|X||A|$. Since $XA\subset A^4$, this implies that $|X||A|\le K|A|$, and hence that $|X|\le K$. Moreover, given $z\in A^3$ the maximality of $X$ implies that there exist $x\in X$ and $a_1,a_2\in A$ such that $za_1=xa_2$, and hence $z=xa_2a_1^{-1}\in XA^2$. In particular, $A^3\subset XA^2$ as required.
\end{proof}

\begin{proof}[Proof of Proposition \ref{prop:vector}]
Proposition \ref{prop:plun} implies that $|4A|\le K^4|A|$. Lemma \ref{lem:cover} therefore implies that there exists a set $X$ of size at most $K^4$ such that $3A\subset X+2A$. This implies by induction that $mA\subset(m-2)X+2A$ for every $m>3$. Writing $\langle B\rangle$ for the subgroup generated by a set $B$, we deduce in particular that $\langle A\rangle\subset\langle X\rangle+2A$, and hence that $|\langle A\rangle|\le|\langle X\rangle||2A|\le m^{K^4}K|A|$, as required.
\end{proof}

\section{Approximate groups}\label{sec:app.grp}

When $G$ is not abelian, Proposition \ref{prop:plun} no longer holds as stated. For example, if $G$ is the free product $H\ast\langle x\rangle$ with $x$ some element of infinite order, and if we take
\begin{equation}\label{eq:plun.exemple}
A=H\cup\{x\},
\end{equation}
then $A^2=H\cup xH\cup Hx\cup\{x^2\}$, hence in particular $|A^2|\le3|A|$. On the other hand, $A^3\supset HxH$ and $|HxH|=|H|^2$, so $|A^3|\ge\frac{1}{4}|A|^2$.

Nonetheless, it turns out that if we replace $|A+A|\le K|A|$ with a slightly stronger hypothesis then we can obtain a conclusion analogous to that of Proposition \ref{prop:plun}. In fact, there are at least two such possible ways in which to strengthen the condition of small doubling. The first is to replace it with the condition of small tripling: an argument of Ruzsa shows that if we assume $|A^3|\le K|A|$ instead of $|A^2|\le K|A|$ for a finite symmetric set $A$ then we may conclude that $|A^m|\le K^{m-2}|A|$ for every $m\in\N$. In other words, unlike small doubling, small tripling permits us to bound the sizes of all of the sets $A^4,A^5,\ldots$.

The second possibility is to replace the condition $|A^2|\le K|A|$ by a property that we have already encountered in both Lemma \ref{lem:cover} and  \eqref{eq:cov.B}: the existence of a set $X$ of bounded size such that $A^2\subset XA$, which easily implies that $A$ has small doubling. It is this condition that underpins the following definition of an approximate subgroup, which is due to Tao.
\begin{definition}\label{def:app.grp}
A subset $A$ of a group $G$ is a \emph{$K$-approximate (sub)group} if it is symmetric and contains the identity and there exists a set $X\subset G$ of size at most $K$ such that $A^2\subset XA$.
\end{definition}
It is easy to see by induction that a $K$-approximate group $A$ satisfies $A^m\subset X^{m-1}A$ for every $m\in\N$, so if $A$ is finite then $|A^m|\le K^{m-1}|A|$ and once again we have an analogue of Proposition \ref{prop:plun}.

In fact, these two conditions -- having small tripling and being an approximate group -- are essentially equivalent for finite sets. We have just noted that if $A$ is a finite $K$-approximate group then $|A^3|\le K^2|A|$, so $A$ has small tripling. Conversely, for a finite symmetric set $A$ satisfying $|A^3|\le K|A|$, the result of Ruzsa shows that $|A^4|\le K^2|A|$, and then Lemma \ref{lem:cover} implies that $A^2$ is a $K^4$-approximate subgroup (we have $A^3\subset XA^2$ by Lemma \ref{lem:cover}, and hence $A^4=A^3A\subset XA^3\subset X^2A^2$).

Note that one advantage of the notion of an approximate subgroup is that it can be applied without modification to arbitrary infinite subsets of groups, for which the notion of small tripling does not in general make sense. Indeed, infinite approximate groups have begun to be studied in certain contexts. However, at the time of writing the theory is far more advanced for finite approximate groups, and we will concentrate on them for the remainder of this article.

When introducing the definition of approximate groups, Tao showed that the study of sets of small doubling essentially reduces to the study of finite approximate groups. First, note that in example \eqref{eq:plun.exemple}, $A$ possesses a large subset that is a $1$-approximate subgroup, namely $H$. Tao showed that this is a general phenomenon, in the sense that there exists $C>0$ such that given any set $|A^2|\le K|A|$ there exists a $K^C$-approximate group $B\subset G$ of size at most $K^C|A|$ such that $A$ is contained in a union of at most $K^C$ left translates of $B$. One may thus replace the hypothesis $|A^2|\le K|A|$ by the hypothesis of being a $K$-approximate subgroup without really losing any generality, whilst gaining the ability to control the sizes of the sets $A^4,A^5,\ldots$.

We close this section by noting that Ruzsa proved Lemma \ref{lem:cover} several years before the introduction of Definition \ref{def:app.grp} by Tao. In that sense, Ruzsa's work can be thought of as a precursor to the notion of approximate group.

\section{Basic properties}
Here are two simple but useful properties of a subgroup $H$ of $G$:
\begin{enumerate}
\item If $\pi:G\to Z$ is a homomorphism then $\pi(H)$ is again a subgroup of $Z$.
\item If $N<G$ is another subgroup then $H\cap N$ is also a subgroup.
\end{enumerate}
It turns out that these properties have approximate analogues for approximate groups and sets of small tripling. For (1), if $A$ is a $K$-approximate subgroup of $G$ and $\pi:G\to Z$ is a homomorphism then it is trivially the case that $\pi(A)$ is a $K$-approximate subgroup of $Z$. Less obviously, an argument of Helfgott shows that if $A$ is a finite symmetric subset of $G$ then
\[
\frac{|\pi(A)^m|}{|\pi(A)|}\le\frac{|A^{m+2}|}{|A|},
\]
so in particular if $|A^3|\le K|A|$ then $|\pi(A)^3|\le K^3|\pi(A)|$. For (2), one can show for example that $A$ and $B$ are finite symmetric subsets of $G$ then
\[
\frac{|A^m\cap B^n|}{|A^2\cap B^2|}\le\frac{|A^{m+1}|}{|A|}\frac{|B^{n+1}|}{|B|}
\]
for every $m,n\ge2$, and in particular if $|A^3|\le K|A|$ and $|B^3|\le L|B|$ then $|(A^2\cap B^2)^3|\le(KL)^{5}|A^2\cap B^2|$. Similarly, if $A$ is a $K$-approximate group and $B$ is an $L$-approximate group then $A^2\cap B^2$ is a $(KL)^3$-approximate group. See \cite[\S2.6]{book} for proofs and generalisations of these assertions.

We saw in the previous section that approximate groups and sets of small tripling are essentially equivalent notions. In this section we have seen that they satisfy the same basic properties, which renders them interchangeable in a number of arguments.

\section{Approximate subgroups of non-abelian groups}
One can generalise the concept of progression to certain non-abelian groups. For example, consider the \emph{Heisenberg group} $H$ defined by
\[
H=
\left(\begin{array}{ccc}
1 & \Z & \Z \\
0 & 1   & \Z \\
0 & 0   &  1
\end{array}\right)=
\left\{\left(\begin{array}{ccc}
1 & n_2   & n_3 \\
0 & 1         & n_1 \\
0 & 0         & 1
\end{array}\right):
n_i\in\Z\right\},
\]
and set
\[
Q=\left\{\left(\begin{array}{ccc}
1 & \ell_1   & \ell_3 \\
0 & 1         & \ell_2 \\
0 & 0         & 1
\end{array}\right):
\begin{array}{c}
|\ell_1|\le L_1,
|\ell_2|\le L_1,
|\ell_3|\le L_1L_2
\end{array}
\right\}.
\]
It is an easy exercise to check that
\[
Q^3\subset\left\{\left(\begin{array}{ccc}
1 & \ell_1   & \ell_3 \\
0 & 1         & \ell_2 \\
0 & 0         & 1
\end{array}\right):
\begin{array}{c}
|\ell_1|\le3L_1,
|\ell_2|\le3L_1,
|\ell_3|\le8L_1L_2
\end{array}
\right\},
\]
and hence that $|Q^3|\le72|Q|$ regardless of the values of $L_1,L_2$.

The key property of $H$ that makes this true is that it is \emph{nilpotent}. To define this, first define the \emph{lower central series} of a group $G$ to be the decreasing sequence of normal subgroups $G=G_1>G_2>\cdots$ defined recursively by setting $G_1=G$ and $G_{n+1}=[G,G_n]$. A group $G$ is then said to be \emph{nilpotent} if there exists $s$ such that $G_{s+1}=\{1\}$. The smallest $s$ for which this holds is said to be the \emph{step} or \emph{class} of $G$. For the Heisenberg group $H$ we have
\[
H_2=\left(\begin{array}{ccc}
1 & 0 & \Z \\
0 & 1   & 0 \\
0 & 0   &  1
\end{array}\right)
\qquad\text{and}\qquad
H_3=\{1\},
\]
so that $H$ is $2$-step nilpotent.

It turns out that we can define a progression in the same way as we did in the Heisenberg group in an arbitrary nilpotent group, as follows.
\begin{definition}
Let $G$ be an $s$-step nilpotent group, let $x_1,\ldots,x_r\in G$, and let $L_1,\ldots,L_r\in\N$. Then we define $P(x;L)\subset G$ to be the set of those elements of $G$ expressible as products of the elements $x_i^{\pm1}$ in which each $x_i$ and its inverse appear at most $L_i$ times between them. We call $P(x;L)$ a \emph{nilprogression} of \emph{rank} $r$ and \emph{step} $s$.
\end{definition}
One can show that if the $L_i$ are large enough in terms of $r$ and $s$ then the nilprogression $P(x;L)$ is a $K$-approximate group of some $K$ depending only on $r$ and $s$.

The `progression' $Q$ is not exactly a nilprogression, but one can check that if we set
\[
x_1=
\left(\begin{array}{ccc}
1 & 0 & 0\\
0 & 1 & 1\\
0 & 0 & 1
\end{array}\right),
\qquad
x_2=
\left(\begin{array}{ccc}
1 & 1 & 0 \\
0 & 1 & 0 \\
0 & 0 & 1
\end{array}\right)
\]
then $P(x;L)\subset Q\subset P(x;5L)$, so $Q$ is roughly equivalent to a nilprogression in some sense. See \cite[Definition 5.6.2]{book} for a generalisation of $Q$ to arbitrary nilpotent groups, and \cite[Proposition 5.6.4]{book} for further details on this rough equivalence.

The following remarkable result of Breuillard, Green and Tao shows that nilprogressions are essentially the most general examples of sets of small doubling.
\begin{theorem}[Breuillard--Green--Tao {\cite[Theorem 2.12]{bgt}}]\label{thm:bgt}
Let $G$ be an arbitrary group and $A\subset G$ a finite subset such that $|A^2|\le K|A|$. Then $G$ contains a subset $P$ containing a finite subgroup $H$ normalised by $P$, such that the image of $P$ in $\langle P\rangle/H$ is a nilprogression of rank at most $r(K)$ and step at most $s(K)$, and such that $|P|\le t(K)|A|$. There also exists a set $X$ of size at most $i(K)$ such that $A\subset XP$.
\end{theorem}
In addition to the general theory of approximate groups, the proof of Theorem \ref{thm:bgt} uses tools from model theory introduced by Hrushovski, and arguments essentially due to Gleason arising from the solution to Hilbert's fifth problem in the 1950s.

The use of an ultrafilter in the model-theoretic arguments means that the proof of Theorem \ref{thm:bgt} gives no explicit bound on $i(K)$. For some applications of approximate groups, notably those to \emph{growth} of groups that we present in Section \ref{sec:apps}, this does not pose a major problem. However, there are also applications of approximate groups, such as to \emph{expansion}, in which it is important to have more explicit results than Theorem \ref{thm:bgt}. Partly for this reason, numerous authors have given proofs of Theorem \ref{thm:bgt} that offer explicit bounds on $i(K)$ in return for restricting attention to certain specific classes of groups. There are such results, for example, in the case of soluble groups, residually nilpotent groups, and certain linear groups. In the next section we will discuss briefly how some of these results for linear groups are used in the construction of expanders.

\section{Applications to growth and expansion in groups}\label{sec:apps}

In this section we describe two of the most spectacular applications of approximate groups. We begin with applications to \emph{growth} of finitely generated groups, a notion that is in turn linked to random walks, geometric group theory and differential geometry. After that we will discuss applications to \emph{expansion}, a notion which appears in several branches of mathematics and has numerous applications, particularly in theoretical computer science.

Let $G$ be a finitely generated group and $S$ a finite symmetric generating subset. The \emph{growth} of $G$ refers to the speed with which the cardinality of the sets $S^1,S^2,\ldots$ grows. It is not difficult to show that if $G$ is \emph{virtually nilpotent}---that is to say, if $G$ contains a nilpotent subgroup of finite index---then there exist $C,d\ge0$ such that $|S^n|\le Cn^d$ for every $n\in\N$. In that case we say that $G$ has \emph{polynomial growth}. A fundamental theorem of Gromov says that the converse also holds: every finitely generated group of polynomial growth is virtually nilpotent.

It turns out that approximate groups can be used to prove Gromov's theorem. In fact, Breuillard, Green and Tao used Theorem \ref{thm:bgt} to prove a refined version of Gromov's theorem. For example, the quantitative statement of Gromov's theorem implicitly requires the generating set to be of bounded cardinality, but in the Breuillard--Green--Tao version this hypothesis is not necessary.

The observation that allows one to reduce Gromov's theorem to Theorem \ref{thm:bgt} is that the condition
\begin{equation}\label{eq:poly.growth+}
|S^n|\le n^d|S|
\end{equation}
implies that there exists $K\ge1$ depending only on $d$, and an integer $m$ satisfying $\sqrt{n}\leq m\leq n$, such that $|S^{2m}|\le K|S^m|$. In other words, \eqref{eq:poly.growth+} implies that there exists $m$ not too small such that $S^m$ is a set of small doubling. Thus, approximate groups appear very naturally in the study of groups of polynomial growth. We refer the reader to \cite[Chapter 11]{book} and the references therein for more details and further applications in this direction.

Another important application of approximate groups is the construction of \emph{expander graphs}. An \emph{expander graph} is a graph that is both sparse and highly connected. Precisely, given a subset $A$ of a finite graph $\Gamma$, we define the \emph{boundary} $\partial A$ of $A$ by setting $\partial A=\{x\in\Gamma\setminus A:(\exists a\in A)(x\sim a)\}$, and we define the \emph{(vertex) Cheeger constant} $h(\Gamma)$ of $\Gamma$ by setting
\[
h(\Gamma)=\min_{|A|\le|\Gamma|/2}\frac{|\partial A|}{|A|}.
\]
Given $\eps>0$ and $d\in\N$, a family $X$ of finite graphs is said to be a \emph{family of $(\eps,d)$-expanders} if $h(\Gamma)\ge\eps$ for every $\Gamma\in X$, if $\sup_{\Gamma\in X}|\Gamma|=\infty$, and if each vertex of each graph in $X$ has degree at most $d$. Note that if a finite graph $\Gamma$ is complete then $h(\Gamma)\ge1$; the upper bound on the degrees rules out this trivial situation, and is the sense in which expanders are sparse.

To see why such graphs are interesting, note that sparsity and high connectivity are both desirable properties of communication and transport networks, yet are intuitively difficult to achieve simultaneously.

One of the objectives, and one of the difficulties, in the theory of expander graphs is their construction. One fruitful approach is based on group theory and the notion of a \emph{Cayley graph}. Given a finitely generated group $G$ and a finite symmetric generating set $S$, the \emph{Cayley graph} $\Gamma(G,S)$ has the elements of $G$ as its vertices, and has $x$ and $y$ joined by an edge if there exists $s\in S$ such that $xs=y$.

It turns out that certain results using techniques from the theory of approximate groups can be applied in the construction of expander Cayley graphs. For example, for $SL_n(\K)$ we have the following theorem, which was announced independently (within four hours of one another!) by Breuillard--Green--Tao and Pyber--Szabo, Helfgott having already treated the cases $d=2,3$ for $\K=\F_p$ with $p$ prime.
\begin{theorem}[{\cite[Theorem 1.5.1]{tao.expansion}}]\label{thm:bgt/ps}
Let $\K$ be a finite field and let $n\ge2$. Let $A$ be a generating set of $SL_n(\K)$. Suppose that $\eps>0$ is small enough in terms of $n$. Then either $|A^3|\ge|A|^{1+\eps}$, or $|A|\ge|SL_n(\K)|^{1-c_n\eps}$, with $c_n$ a certain constant depending only on $n$.
\end{theorem}
It turns out that using Theorem \ref{thm:bgt/ps} and an ingenious argument of Bourgain and Gamburd one can show that certain Cayley graphs of $SL_n(\F_p)$ are expander graphs. For further details on this argument and its history the reader can consult Tao's book \cite{tao.expansion}.

\end{document}